\documentclass[]{amsart}

\usepackage[english]{babel}
\usepackage[utf8]{inputenc}
\usepackage{paralist}
\usepackage{amsmath, amsfonts, amssymb, amsthm}
\usepackage{color}
\usepackage[normalem]{ulem}
\usepackage{tikz}

\newcommand{\AAA}{\mathbb{A}}
\newcommand{\CC}{\mathbb{C}}

\newcommand{\GG}{\mathbb{G}}
\newcommand{\NN}{\mathbb{N}}

\newcommand{\QQ}{\mathbb{Q}}

\newcommand{\ZZ}{\mathbb{Z}}

\newcommand{\Hc}{\mathcal{H}}

\newcommand{\Oc}{\mathcal{O}}
\newcommand{\Pc}{\mathcal{P}}

\newcommand{\set}[1]{\left\{ #1 \right\}}
\newcommand{\setb}[1]{\left( #1 \right)}
\newcommand{\abs}[1]{\left| #1 \right|}

\newcommand{\m}{\mathrm{m}}
\newcommand{\ibf}{\mathbf{i}}
\newcommand{\ubf}{\mathbf{u}}
\newcommand{\Xbf}{\mathbf{X}}
\newcommand{\xbf}{\mathbf{x}}

\DeclareMathOperator{\Gal}{Gal}

\newtheorem{mymasterthm}{notForUse}
\theoremstyle{definition}

\theoremstyle{plain}
\newtheorem{mylemma}[mymasterthm]{Lemma}
\newtheorem{mythm}[mymasterthm]{Theorem}

\allowdisplaybreaks

\title[Recurrences as Coefficients of Polynomials]{Integral Zeros of a Polynomial with Linear Recurrences as Coefficients}

\author[C. Fuchs]{Clemens Fuchs}
\author[S. Heintze]{Sebastian Heintze}
\thanks{Supported by Austrian Science Fund (FWF): I4406.}

\address{University of Salzburg\newline
	\indent Department of Mathematics\newline
	\indent Hellbrunnerstr. 34 \newline
	\indent A-5020 Salzburg, Austria}
\email{clemens.fuchs@sbg.ac.at, sebastian.heintze@sbg.ac.at}

\begin{document}
	
	\maketitle
	

	\begin{abstract}
		Let $ K $ be a number field, $ S $ a finite set of places of $ K $, and $ \Oc_S $ be the ring of $ S $-integers. Moreover, let
		\begin{equation*}
			G_n^{(0)} Z^d + \cdots + G_n^{(d-1)} Z + G_n^{(d)}
		\end{equation*}
		be a polynomial in $ Z $ having simple linear recurrences of integers evaluated at $ n $ as coefficients. Assuming some technical conditions we give a description of the zeros $ (n,z) \in \NN \times \Oc_S $ of the above polynomial. We also give a result in the spirit of Hilbert irreducibility for such polynomials.
	\end{abstract}
	
	\section{Introduction}
	
	In the present paper we consider a special kind of polynomial-exponential Diophantine equations.
	More specifically we are interested in zeros of a polynomial in one variable where the coefficients of this polynomial are simple linear recurrences taking only integral values evaluated at the same $ n \in \NN $.
	In other words we consider the equation
	\begin{equation}
		\label{p5-eq:introrecurr}
		G_n^{(0)} Z^d + \cdots + G_n^{(d-1)} Z + G_n^{(d)} = 0,
	\end{equation}
	where $ G_n^{(i)} $ are sequences of integers satisfying a simple linear recurrence relation.
	Clearly, this is a family of polynomials parametrized by an exponential variable $ n $; for any $ n \in \NN $ we thus have to deal with a polynomial with integers coefficients.
	The main focus lies on solutions $ (n,z) \in \NN \times \Oc_S $ of \eqref{p5-eq:introrecurr}.
	Here we denote by $\NN$ the set of positive integers and by $ \Oc_S $ the ring of $ S $-integers for a number field $ K $ and a finite set $ S $ of absolute values on $ K $, containing all archimedean ones.
	
	Every simple linear recurrence sequence $ G_n $ can be written in its Binet representation
	\begin{equation*}
		G_n = b_1 \beta_1^n + \cdots + b_t \beta_t^n.
	\end{equation*}
	The $ \beta_i $ are called the (characteristic) roots and the $ b_i $ are called the coefficients of the recurrence $ G_n $. Both, the characteristic roots as well as the coefficients of the recurrence, are elements of a finite extension $ L $ of $ K $. For basic knowledge on recurrences we refer to \cite{schmidt-2003}.
	It follows that we can choose a common numbering $ \beta_1,\ldots,\beta_r $ of all occurring characteristic roots and rewrite \eqref{p5-eq:introrecurr} as
	\begin{equation}
		\label{p5-eq:intropoly}
		a_0(\beta_1^n,\ldots,\beta_r^n) Z^d + \cdots + a_{d-1}(\beta_1^n,\ldots,\beta_r^n) Z + a_d(\beta_1^n,\ldots,\beta_r^n) = 0
	\end{equation}
	with linear polynomials $ a_0(X_1,\ldots,X_r), \ldots, a_d(X_1,\ldots,X_r) $.
	Therefore the problem translates into an equation given by a (rather special lacunary) polynomial for which we seek integral solutions in $ \GG_{\m}^r \times \AAA^1 $.
	Conversely, every hypersurface in $ \GG_{\m}^r \times \AAA^1 $ can be written in the form
	\begin{equation*}
		a_0(X_1,\ldots,X_r) Z^d + \cdots + a_{d-1}(X_1,\ldots,X_r) Z + a_d(X_1,\ldots,X_r) = 0
	\end{equation*}
	for (not necessarily linear) polynomials $ a_j(X_1,\ldots,X_r) $.
	The integral points on such a hypersurface are the elements of $ (\Oc_S^*)^r \times \Oc_S $ which satisfy the given equation.
	If the equation is monic in $ Z $ or the leading coefficient is a constant times a monomial in $ X_1,\ldots,X_r $, then it describes a finite cover $ W \rightarrow \GG_{\m}^r $ given by projection on the first $ r $ components. Here we specialize to a $ 1 $-parameter subgroup of $ \GG_{\m}^r $ for which \eqref{p5-eq:intropoly} is the typical description.
	(Compare e.g. with \cite{fuchs-mantova-zannier-2018} where all regular maps $ \GG_{\m} \rightarrow W $, i.e. function field integral points, of the finite cover $ W \rightarrow \GG_{\m}^r $ are described.)
	
	The first author together with Scremin considered in \cite{fuchs-scremin-2004} integer solutions of \eqref{p5-eq:introrecurr} using the recurrence point of view, i.e. viewing \eqref{p5-eq:introrecurr} as polynomial-exponential Diophantine equation. Their result was heavily based on Corvaja and Zannier's results concerning the equation $f(G_n,Z)=0$ handled in \cite{corvaja-zannier-2002-1}. We also mention \cite{corvaja-zannier-2000} in which similar conditions and results appear. The main and most restrictive technical condition is the existence of ``dominant roots''. Without this condition one can currently expect only weaker results; see e.g. \cite{corvaja-zannier-2002-2} for a result without the dominant root condition and \cite{zannier-2010} for a somewhat weaker, but still useful result.
	Here we will analyze solutions in $ S $-integers and (most of the time) make use of the formulation in the polynomial point of view given by \eqref{p5-eq:intropoly}.
	
	\section{Results}
	
	Within this section we state our results that will be proven in the present paper.
	The first one looks a little bit technical, but it is an essential, intermediate result that will be used to prove the second one.
	We remark that a given set $ S $ of places of $ K $ can be extended such that it satisfies the properties required in the theorems. This only makes the statements stronger.
	
	\begin{mythm}
		\label{p5-thm:mainresult}
		Let $ K $ be a number field and $ g \in K[X_1,\ldots,X_r,Z] $ a polynomial which can be written in the form
		\begin{equation*}
			g(X_1,\ldots,X_r,Z) = a_0(X_1,\ldots,X_r) Z^d + \cdots + a_d(X_1,\ldots,X_r)
		\end{equation*}
		for linear polynomials $ a_0(X_1,\ldots,X_r), \ldots, a_d(X_1,\ldots,X_r) $.
		Furthermore, let $ \widetilde{g} \in K[X_1,\ldots,X_r,\widetilde{Z}] $ be the polynomial given by the equation
		\begin{equation*}
			\widetilde{g}(X_1,\ldots,X_r,a_0(X_1,\ldots,X_r)Z) = a_0(X_1,\ldots,X_r)^{d-1} g(X_1,\ldots,X_r,Z).
		\end{equation*}
		Now assume that either $ a_0(0,\ldots,0) \neq 0 $ and $ g(0,\ldots,0,Z) $ has no multiple zero as a polynomial in $ Z $, or $ a_0(0,\ldots,0) = 0 $ and $ \widetilde{g}(0,\ldots,0,\widetilde{Z}) $ has no multiple zero as a polynomial in $ \widetilde{Z} $.
		Moreover, let $ \gamma_1,\ldots,\gamma_r \in K^* $ such that $ \abs{\gamma_i} < 1 $ for all $ 1 \leq i \leq r $ and such that no ratio $ \gamma_i / \gamma_j $ for $ i \neq j $ is a root of unity.
		Assume that $ S $ is a finite set of places of $ K $, containing all archimedean ones, and such that $ \gamma_1,\ldots,\gamma_r $ and all non-zero coefficients of $ a_i(X_1,\ldots,X_r) $ for $ i=0,\ldots,d $ are $ S $-units.
		
		Then there are finitely many cosets $ u_1H_1, \ldots, u_tH_t \subseteq \GG_{\m}^{r} $ and for each coset $ u_iH_i $ a polynomial $ P_i $ in $ r $ unknowns such that the following holds:
		For each solution $ (n,z) \in \NN \times \Oc_S $ of $ g(\gamma_1^n,\ldots,\gamma_r^n,z) = 0 $ with $ z \neq 0 $ and $ n $ large enough, there exists an index $ i $ such that $ (\gamma_1^n,\ldots,\gamma_r^n) \in u_iH_i $ and $ z' = P_i(\gamma_1^n,\ldots,\gamma_r^n) $, where $ z'=z $ in the case $ a_0(0,\ldots,0) \neq 0 $ and $ z' = a_0(\gamma_1^n,\ldots,\gamma_r^n) z $ if $ a_0(0,\ldots,0) = 0 $, respectively.
	\end{mythm}
	
	It is clear that the same result holds true if we assume that $ a_0(0,\ldots,0) \neq 0 $ and $ \widetilde{g}(0,\ldots,0,\widetilde{Z}) $ has no multiple zero as a polynomial in $ \widetilde{Z} $, as a third alternative for the two options given in the theorem. The proof is completely the same.

	Let us emphasize that this result, as the result in \cite{fuchs-scremin-2004}, goes in the same direction as Corvaja and Zannier's Theorem 2 of \cite{corvaja-zannier-2002-1} and uses similar assumptions, though the results are not quite equal (in the sense that our result does not directly follow from theirs and vice-versa). Moreover, we completely build on the methods developed by them (see Theorem \ref{p5-thm:infinmanycosets} below, which is our main tool, or, e.g., their book \cite{corvaja-zannier-2018} summarizing the developments and results).
	
	Our second theorem considers now polynomial equations with simple linear recurrences as coefficients. The first author and Scremin considered in \cite{fuchs-scremin-2004} a similar situation. The difference is that they were only looking for solutions in $ \NN \times \ZZ $ and that they imposed other (stronger) restrictions on the characteristic roots.
	
	\begin{mythm}
		\label{p5-thm:solofpolyexp}
		Let $ K, g, \widetilde{g}, \gamma_1, \ldots, \gamma_r $ and $ S $ be as in Theorem \ref{p5-thm:mainresult}.
		Then there are finitely many linear recurrences $ R_1(n), \ldots, R_s(n) $ with algebraic roots and algebraic coefficients, arithmetic progressions $ \Pc_1, \ldots, \Pc_s $, as well as finite sets $ M $ and $ N $ such that the set $ L $ of solutions $ (n,z) \in \NN \times \Oc_S $ of the equation $ g(\gamma_1^n,\ldots,\gamma_r^n,z) = 0 $ can be described by
		\begin{equation*}
			L = \bigcup_{j=1}^{s} \set{(n,R_j(n)) : n \in \Pc_j, R_j(n) \in \Oc_S} \cup \set{(n,z) : n \in N, z \in \Oc_S} \cup M.
		\end{equation*}
	\end{mythm}
	
    We also exploit the reducibility of such a polynomial. This follows similar to Theorem \ref{p5-thm:mainresult} and by standard arguments in the theory of Hilbert irreducibility (see e.g. \cite{lang-1983}); compare also with the application of such results on Diophantine equations with power sums to universal Hilbert sets in \cite{corvaja-zannier-1998} and \cite{corvaja-zannier-2002-1} and with \cite{zannier-2010}.
	
    \begin{mythm}
	    \label{p5-thm:red}
    	Let $ K, g, \gamma_1, \ldots, \gamma_r $ and $ S $ be as in Theorem \ref{p5-thm:mainresult}. Moreover, assume that $ g $ is monic as a polynomial in $ Z $, i.e. $ a_0(X_1,\ldots,X_r) = 1 $. Then $ g(\gamma_1^n,\ldots,\gamma_r^n,Z) $ is reducible in $ K[Z] $ for infinitely many $ n \in \NN $ if and only if there exist monic polynomials $ h_1(n,Z), h_2(n,Z) $, whose coefficients are linear recurrences with algebraic characteristic roots and algebraic coefficients, and an arithmetic progression $ \Pc $ such that $ g(\gamma_1^n,\ldots,\gamma_r^n,Z) = h_1(n,Z) h_2(n,Z) $ is a factorization in $ K[Z] $ for all $ n \in \Pc $.
    \end{mythm}
	
	In the case that the polynomial $ g $ is not monic in $ Z $, one can use the transformation to $ \widetilde{g} $ written down in Theorem \ref{p5-thm:mainresult}. Then $ \widetilde{g} $ is monic in $ \widetilde{Z} $ and Theorem \ref{p5-thm:red} can be applied to it.
	Going back to $ g $ then yields the result that $ g(\gamma_1^n,\ldots,\gamma_r^n,Z) $ is reducible in $ K[Z] $ for infinitely many $ n \in \NN $ if and only if there exist polynomials $ h_1(n,Z), h_2(n,Z) $, whose coefficients are linear recurrences with algebraic characteristic roots and algebraic coefficients, and an arithmetic progression $ \Pc $ such that $ a_0(\gamma_1^n,\ldots,\gamma_r^n)^{d-1} g(\gamma_1^n,\ldots,\gamma_r^n,Z) = h_1(n,Z) h_2(n,Z) $ is a factorization in $ K[Z] $ for all $ n \in \Pc $.
	
	We remark that generic decompositions, as they occur in the statement of Theorem \ref{p5-thm:red}, can be computed by Lemma 2.1 in \cite{corvaja-zannier-2002-2}.
	
    It follows, under the conditions we work in, that if $ g(\gamma_1^n,\ldots,\gamma_r^n,Z) $ is irreducible as a polynomial in $ Z $ over the ring of $ K $-valued power sums (or, more general, the Hadamard ring of linear recurrences in $ K $), then it cannot be reducible in $ K[Z] $ for infinitely many $ n \in \NN $. As usual one may deduce that all decompositions can be described in ``finite terms'' coming from finitely many generic decompositions of $ g(\gamma_1^n,\ldots,\gamma_r^n,Z) $ over the ring whose coefficients are linear recurrences in $ K $ with finitely many exceptions.
	
	In contrast to \cite{fuchs-scremin-2004} we have now a much more powerful tool in our hands; instead of applying the Subspace theorem directly (as done in \cite{fuchs-scremin-2004}), we can apply Theorem \ref{p5-thm:infinmanycosets} below (that in turn follows from the Subspace theorem). This leads to a much quicker proof. A few words on the strategy of proof can be found at the beginning of section \ref{p5-sec:proofsec}.
	
	\section{Preliminaries}
	
	In our proofs we will need some auxiliary results which are listed below. The first one is a result of Schmidt about the zero multiplicity of linear recurrence sequences and can be found in \cite{schmidt-1999}:
	\begin{mythm}
		\label{p5-thm:zerobound}
		Suppose that $ (G_n)_{n \in \ZZ} $ is a non-degenerate linear recurrence sequence of complex numbers, whose characteristic polynomial has $ k $ distinct roots of multiplicity $ \leq a $. Then the number of solutions $ n \in \ZZ $ of the equation
		\begin{equation*}
			G_n = 0
		\end{equation*}
		can be bounded above by
		\begin{equation*}
			c(k,a) = e^{(7k^a)^{8k^a}}.
		\end{equation*}
	\end{mythm}
	
	Given two vectors $ \xbf = (x_1,\ldots,x_k) $ and $ \ibf = (i_1,\ldots,i_k) $, we use the abbreviation $ \xbf^{\ibf} = x_1^{i_1} \cdots x_k^{i_k} $.
	For a number field $ K $ and a finite set $ S $ of absolute values (containing the archimedean ones) we denote the absolute logarithmic Weil height of an element $ x \in K^* $ by
	\begin{equation*}
		h(x) = \sum_{v} \max \setb{0, \log \abs{x}_v}
	\end{equation*}
	and the $ S $-height by
	\begin{equation*}
		h_S(x) = \sum_{v \notin S} \max \setb{0, \log \abs{x}_v},
	\end{equation*}
	where the sums are taken over all places of $ K $.
	Moreover, for a vector $ \xbf $ we denote by $ h(\xbf) $ the usual projective logarithmic height and by $ \widehat{h}(\xbf) $ the sum of the heights of the coordinates of $ \xbf $.
	We use the Landau symbols $ O $ and $ o $ in the usual way, i.e. $ f(n) = O(g(n)) $ if there is a constant $ C $ such that $ f(n) \leq Cg(n) $ for all $ n $, and $ f(n) = o(g(n)) $ if $ f(n)/g(n) \rightarrow 0 $ as $ n $ goes to infinity.
	The following theorem is an essential part of our proof. It is proven as Theorem 1 in \cite{corvaja-zannier-2005} by Corvaja and Zannier:
	\begin{mythm}
		\label{p5-thm:infinmanycosets}
		Let $ f(\Xbf) = \sum_{\ibf} a_{\ibf} \Xbf^{\ibf} $ be a power series with algebraic coefficients in $ \CC_{\nu} $ converging in a neighborhood of the origin in $ \CC_{\nu}^k $.
		Let $ S $ be a finite set of absolute values of $ K $ containing the archimedean ones.
		Let $ \xbf_n = (x_{n1},\ldots,x_{nk}) $ $ (n=1,2,\ldots) $ be a sequence in $ {K^*}^k $, tending to zero in $ K_{\nu}^k $ and such that $ f(\xbf_n) $ is defined and belongs to $ K $.
		Suppose that:
		\begin{enumerate}[1)]
			\item For $ i = 1,\ldots,k $ we have $ h_S(x_{ni}) + h_S(x_{ni}^{-1}) = o(h(x_{ni})) $ as $ n \rightarrow \infty $.
			\item $ \widehat{h}(\xbf_n) = O(-\log (\max_i \abs{x_{ni}}_{\nu})) $.
			\item $ h_S(f(\xbf_n)) = o(h(\xbf_n)) $.
			\item $ h(f(\xbf_n)) = O(h(\xbf_n)) $.
		\end{enumerate}
		Then there exists a finite number of cosets $ \ubf_1H_1, \ldots, \ubf_tH_t \subseteq \GG_{\m}^k $ such that $ \set{\xbf_n}_{n \in \NN} $ $ \subseteq \bigcup_{i=1}^{t} \ubf_iH_i $ and such that, for $ i=1,\ldots,t $, the restriction of $ f(\Xbf) $ to $ \ubf_iH_i $ coincides with a polynomial in $ K[\Xbf] $.
	\end{mythm}
	
	The logarithmic height defined above has some properties which we list in the next lemma. For a proof of these properties we refer to \cite{zannier-2009}:
	\begin{mylemma}
		\label{p5-lemma:heightprop}
		For the above defined logarithmic height the following holds:
		\begin{enumerate}[a)]
			\item $ h(x) \geq 0 $.
			\item $ h(1/x) = h(x) $ and $ h(x^m) = \abs{m} \cdot h(x) $ for all $ m \in \ZZ $.
			\item $ h(x_1 \cdots x_r) \leq h(x_1) + \cdots + h(x_r) $.
			\item $ h(x_1 + \cdots + x_r) \leq h(1 : x_1 : \ldots : x_r) + \log r $.
			\item $ h(x^{\sigma}) = h(x) $ for all $ \sigma \in \Gal(\overline{\QQ}/\QQ) $.
			\item $ \max_{i=1,\ldots,r} h(x_i) \leq h(1 : x_1 : \ldots : x_r) \leq h(x_1) + \cdots + h(x_r) $.
		\end{enumerate}
	\end{mylemma}
	
	Another result that we will need later on is the Hadamard Quotient Theorem. When using this, we denote by $ \Hc(K) $ the Hadamard ring, i.e. the set of sequences in $ K $ satisfying a linear recurrence relation. Let us first state the following version from \cite{ferretti-zannier-2007}:
	\begin{mythm}[\textbf{Hadamard Quotient Theorem}]
		\label{p5-thm:hadamardquotient}
		Let $ K $ be a field of characteristic zero and let $ b(n), c(n) \in \Hc(K) $.
		Let $ (a_n) $ be a sequence whose elements are in a subring $ R $ of $ K $ which is finitely generated over $ \ZZ $, and suppose that $ a_n = \frac{b(n)}{c(n)} $ whenever the quotient is defined.
		Then there exists an element $ a(n) \in \Hc(K) $ such that $ a(n) = a_n $ for every $ n $ such that $ c(n) \neq 0 $.
	\end{mythm}
	
	Corvaja and Zannier proved in \cite{corvaja-zannier-2002-2} a variant of the Hadamard Quotient Theorem:
	If the characteristic roots of the simple linear recurrences $ b(n) $ and $ c(n) $ generate together a torsion-free multiplicative group and if for infinitely many $ n $ we have $ c(n) \neq 0 $ as well as $ \frac{b(n)}{c(n)} \in R $, then the sequence $ n \mapsto \frac{b(n)}{c(n)} $ is a linear recurrence sequence.
	We will use this variant in our proof.
	
	Last but not least we will make use of a suitable version of the Implicit Function Theorem. For more information about this and other versions of the Implicit Function Theorem we refer to \cite{krantz-parks-2002-1} and \cite{krantz-parks-2002-2}.
	In the formulation of the theorem we use for a multiindex $ \alpha = (\alpha_1,\ldots,\alpha_r) \in \NN^r $ the notation
	\begin{equation*}
		\abs{\alpha} = \alpha_1 + \cdots + \alpha_r
	\end{equation*}
	and write $ 0 $ as a shortcut for $ (0,\ldots,0) $:
	\begin{mythm}[\textbf{Implicit Function Theorem}]
		\label{p5-thm:implicitfunc}
		Suppose the power series
		\begin{equation*}
			F(x_1,\ldots,x_r,y) = \sum_{\abs{\alpha} \geq 0, k \geq 0} a_{\alpha,k} x_1^{\alpha_1} \cdots x_r^{\alpha_r} y^k
		\end{equation*}
		is absolutely convergent for $ \abs{x_1} + \cdots + \abs{x_r} \leq R_1 $, $ \abs{y} \leq R_2 $. If
		\begin{equation*}
			a_{0,0} = 0 \text{ and } a_{0,1} \neq 0
		\end{equation*}
		then there exist $ r_0 > 0 $ and a power series
		\begin{equation}
			\label{p5-eq:implfuncpowser}
			f(x_1,\ldots,x_r) = \sum_{\abs{\alpha} > 0} c_{\alpha} x_1^{\alpha_1} \cdots x_r^{\alpha_r}
		\end{equation}
		such that \eqref{p5-eq:implfuncpowser} is absolutely convergent for $ \abs{x_1} + \cdots + \abs{x_r} \leq r_0 $ and
		\begin{equation*}
			F(x_1,\ldots,x_r,f(x_1,\ldots,x_r)) = 0.
		\end{equation*}
		Moreover, if the coefficients of $ F $ are algebraic, then the coefficients of $ f $ are also algebraic.
	\end{mythm}
	
	This version of the Implicit Function Theorem was used as well in \cite{fuchs-scremin-2004} by the first author and Scremin. As they did there, we emphasize here again that the statement holds in a more general form:
	Suppose that $ F(x_1,\ldots,x_r,y) $ converges absolutely for $ \abs{x_1} + \cdots + \abs{x_r} \leq R_1 $ and that $ \abs{y-y_0} \leq R_2 $ for some $ y_0 \in \overline{\QQ} $ with $ F(0,\ldots,0,y_0) = 0 $. Then under the assumption that
	\begin{equation*}
		\frac{\partial F}{\partial y} (0,\ldots,0,y_0) \neq 0,
	\end{equation*}
	the conclusion is that there exists a power series
	\begin{equation*}
		f(x_1,\ldots,x_r) = \sum_{\abs{\alpha} \geq 0} c_{\alpha} x_1^{\alpha_1} \cdots x_r^{\alpha_r}
	\end{equation*}
	for which the same as above holds.
	
	\section{Proofs}
	\label{p5-sec:proofsec}
	
	In the proof of Theorem \ref{p5-thm:mainresult} we will use the following statement. Since this inequality is of independent applicability we state and prove it separately rather than tacitly claim and prove it within the big proof:
	\begin{mylemma}
		\label{p5-lemma:heightofzero}
		Let $ K $ be a number field and $ f(X) = b_0X^d + b_1X^{d-1} + \cdots + b_d \in K[X] $ a polynomial with $ b_0 \neq 0 $. Assume that $ \xi \in K $ is a zero of $ f $, i.e. $ f(\xi) = 0 $. Then we have for the logarithmic height of $ \xi $ the bound
		\begin{equation*}
			h(\xi) \leq h(1:b_0:b_1:\ldots:b_d) + \log d.
		\end{equation*}
	\end{mylemma}
	
	The proof of such an upper bound on the height of a zero of a polynomial is given as an exercise in \cite{zannier-2009}. For the sake of completeness we write it down in detail:
	\begin{proof}
		During these calculations we will apply the rules from Lemma \ref{p5-lemma:heightprop} several times without explicitly mentioning it.
		First we find the following upper bound for polynomial expressions:
		\begin{align*}
			h(b_d + b_{d-1}\xi &+ \cdots + b_1\xi^{d-1}) \leq \\
			&\leq h(1 : b_d : b_{d-1}\xi : \ldots : b_1\xi^{d-1}) + \log d \\
			&= \log \prod_{v} \max \setb{1, \abs{b_d}_v, \ldots, \abs{b_1\xi^{d-1}}_v} + \log d \\
			&\leq \log \prod_{v} \max \setb{1, \abs{b_d}_v, \ldots, \abs{b_1}_v} \cdot \max \setb{1, \abs{\xi}_v^{d-1}} + \log d \\
			&= \log \prod_{v} \max \setb{1, \abs{b_d}_v, \ldots, \abs{b_1}_v} + \log \prod_{v} \max \setb{1, \abs{\xi}_v^{d-1}} + \log d \\
			&= h(1:b_1:\ldots:b_d) + (d-1) \cdot h(\xi) + \log d.
		\end{align*}
		Since $ f(\xi) = 0 $ and $ b_0 \neq 0 $ we can write
		\begin{equation*}
			\xi^d = - \left( \frac{b_d}{b_0} + \frac{b_{d-1}}{b_0} \xi + \cdots + \frac{b_1}{b_0} \xi^{d-1} \right)
		\end{equation*}
		and thus
		\begin{align*}
			d \cdot h(\xi) &= h(\xi^d) = h\left( \frac{b_d}{b_0} + \frac{b_{d-1}}{b_0} \xi + \cdots + \frac{b_1}{b_0} \xi^{d-1} \right) \\
			&\leq h \left( 1:\frac{b_1}{b_0}:\ldots:\frac{b_d}{b_0} \right) + (d-1) \cdot h(\xi) + \log d.
		\end{align*}
		This yields
		\begin{align*}
			h(\xi) &\leq h \left( 1:\frac{b_1}{b_0}:\ldots:\frac{b_d}{b_0} \right) + \log d \\
			&= h(b_0:b_1:\ldots:b_d) + \log d \\
			&\leq h(1:b_0:b_1:\ldots:b_d) + \log d.
		\end{align*}
	\end{proof}
	
	Before writing down the proof of Theorem \ref{p5-thm:mainresult} in detail, we give a short overview.
	Considering an infinite sequence of solutions $ (n,z) $ for the equation in question we will show that the $ z $-component must be bounded. Afterwards we will calculate a bound on the height of the $ z $-component. These are preparations needed when applying in the sequel the Implicit Function Theorem and finally Theorem \ref{p5-thm:infinmanycosets}.
	
	\begin{proof}[Proof of Theorem \ref{p5-thm:mainresult}]
		Let $ K, g, \widetilde{g}, \gamma_1, \ldots, \gamma_r $ and $ S $ be as in the theorem. We write
		\begin{equation*}
			g(\gamma_1^n,\ldots,\gamma_r^n,Z) = a_0(\gamma_1^n,\ldots,\gamma_r^n) Z^d + \cdots + a_d(\gamma_1^n,\ldots,\gamma_r^n)
		\end{equation*}
		and
		\begin{equation*}
			\widetilde{g}(\gamma_1^n,\ldots,\gamma_r^n,\widetilde{Z}) = \widetilde{a_0}(\gamma_1^n,\ldots,\gamma_r^n) \widetilde{Z}^d + \cdots + \widetilde{a_d}(\gamma_1^n,\ldots,\gamma_r^n).
		\end{equation*}
		Since no ratio $ \gamma_i/\gamma_j $ is a root of unity, by Theorem \ref{p5-thm:zerobound} for $ n $ large enough we have $ a_i(\gamma_1^n,\ldots,\gamma_r^n) \neq 0 $ for all $ 1 \leq i \leq d $.
		As the $ \widetilde{a_j}(\gamma_1^n,\ldots,\gamma_r^n) $ arise by construction as products of the $ a_i(\gamma_1^n,\ldots,\gamma_r^n) $ they are non-zero as well for large $ n $ and all $ 1 \leq j \leq d $.
		Thus we will assume from here on that $ n $ is large enough such that all $ a_i $ (and all $ \widetilde{a_j} $) are non-zero.
		
		At this position we are going to split the proof into two cases: Let us assume that $ a_0(0,\ldots,0) \neq 0 $ and that $ g(0,\ldots,0,Z) $ has only simple zeros. In this case we work only with $ g $ and do not need the transformation $ \widetilde{g} $.
		The other case, when $ a_0(0,\ldots,0) = 0 $ and $ \widetilde{g}(0,\ldots,0,\widetilde{Z}) $ has only simple zeros, works in the completely same way considering $ \widetilde{g} $ instead of $ g $ with the transformation $ \widetilde{z} = a_0(\gamma_1^n,\ldots,\gamma_r^n) z $ and recognizing the fact that $ \widetilde{g}(\gamma_1^n,\ldots,\gamma_r^n,\widetilde{z}) $ is monic in $ \widetilde{z} $, i.e. $ \widetilde{a_0}(\gamma_1^n,\ldots,\gamma_r^n) = 1 \neq 0 $. Hence we will write down only the first case in detail.
		
		Consider now an infinite sequence $ ((n,z_n))_{n \in W} $ of solutions of the equation
		\begin{equation}
			\label{p5-eq:theeq}
			g(\gamma_1^n,\ldots,\gamma_r^n,z) = 0
		\end{equation}
		in $ (n,z) \in \NN \times \Oc_S $ with $ z \neq 0 $, where $ W $ is an infinite subset of $ \NN $.
		Since for fixed $ n $ there are at most $ d $ possible values for $ z $, all solutions are contained in finitely many such sequences. Therefore we can restrict our considerations to one of them.
		To simplify the notation we shall often write $ a_i(n) $ instead of $ a_i(\gamma_1^n,\ldots,\gamma_r^n) $.
		
		Our first step is to prove that the sequence $ (z_n) $ is bounded. Since $ z_n $ is a solution of \eqref{p5-eq:theeq} we have
		\begin{equation}
			\label{p5-eq:zerozn}
			a_0(n)z_n^d + \cdots + a_d(n) = 0.
		\end{equation}
		Using $ z_n \neq 0 $ this is equivalent to
		\begin{equation*}
			a_0(n)z_n = - a_1(n) - \cdots - a_d(n)z_n^{-(d-1)}.
		\end{equation*}
		For $ \abs{z_n} > 1 $ this yields
		\begin{equation*}
			\abs{a_0(n)z_n} \leq \abs{a_1(n)} + \cdots + \abs{a_d(n)}.
		\end{equation*}
		Hence we end up with the upper bound
		\begin{equation*}
			\abs{z_n} \leq \max \setb{1, \frac{\abs{a_1(n)} + \cdots + \abs{a_d(n)}}{\abs{a_0(n)}}}.
		\end{equation*}
		As $ a_0(0,\ldots,0) \neq 0 $ the denominator $ \abs{a_0(n)} $ is bounded away from zero, which gives us together with $ \abs{\gamma_i} < 1 $ the boundedness of the sequence $ (z_n) $.
		
		Based on the estimate
		\begin{align*}
			\abs{g(0,\ldots,0,z_n)} &= \abs{g(0,\ldots,0,z_n) - g(\gamma_1^n,\ldots,\gamma_r^n,z_n)} \\
			&\leq \sum_{i=0}^{d} \underbrace{\abs{a_i(0,\ldots,0) - a_i(\gamma_1^n,\ldots,\gamma_r^n)}}_{\overset{n \rightarrow \infty}{\longrightarrow} 0} \cdot \abs{z_n}^{d-i}
		\end{align*}
		the boundedness of $ z_n $ also implies that $ g(0,\ldots,0,z_n) \rightarrow 0 $ as $ n \rightarrow \infty $. Thus the $ z_n $ lie in the union of arbitrary small neighborhoods of the solutions of $ g(0,\ldots,0,z) = 0 $ for $ n $ large enough.
		Thus we can split the sequence into finitely many subsequences and consider in what follows only an infinite sequence $ (z_n) $ which converges to a solution $ z_* $ of $ g(0,\ldots,0,z) = 0 $.
		
		Before going on we will pause for a moment and derive an upper bound for the logarithmic height of $ z_n $ that will be useful later on.
		Remembering equation \eqref{p5-eq:zerozn} we get by Lemma \ref{p5-lemma:heightofzero} and Lemma \ref{p5-lemma:heightprop} the inequality
		\begin{align*}
			h(z_n) &\leq h(1 : a_0(n) : \ldots : a_d(n)) + \log d \\
			&\leq h(a_0(n)) + \cdots + h(a_d(n)) + \log d.
		\end{align*}
		Now we need a bound on $ h(a_i(n)) $. Using Lemma \ref{p5-lemma:heightprop} once again yields
		\begin{align*}
			h(a_i(n)) &= h \left( \sum_{k_1,\ldots,k_r} \lambda_{k_1,\ldots,k_r}^{(i)} \gamma_1^{nk_1} \cdots \gamma_r^{nk_r} \right) \\
			&\leq \sum_{k_1,\ldots,k_r} h \left( \lambda_{k_1,\ldots,k_r}^{(i)} \gamma_1^{nk_1} \cdots \gamma_r^{nk_r} \right) + C_{i,0} \\
			&\leq \sum_{k_1,\ldots,k_r} \left( h \left( \lambda_{k_1,\ldots,k_r}^{(i)} \right) + k_1 h(\gamma_1^n) + \cdots + k_r h(\gamma_r^n) \right) + C_{i,0} \\
			&= C_{i,0}' + C_{i,1} h(\gamma_1^n) + \cdots + C_{i,r} h(\gamma_r^n) \\
			&\leq C_{i,0}' + (C_{i,1} + \cdots + C_{i,r}) \cdot \max_{j=1,\ldots,r} h(\gamma_j^n) \\
			&\leq C_{i,0}' + C_i' \cdot h(1:\gamma_1^n:\ldots:\gamma_r^n)
		\end{align*}
		where all the constants depend only on $ a_i(X_1,\ldots,X_r) $. In particular, they are independent of $ n $.
		Combining the last two inequality chains we end up with the upper bound
		\begin{equation}
			\label{p5-eq:boundhzn}
			h(z_n) \leq C_1 + C_2 \cdot h(1:\gamma_1^n:\ldots:\gamma_r^n).
		\end{equation}
		
		Note that the condition that $ g(0,\ldots,0,Z) $ has no multiple zero as a polynomial in $ Z $ is equivalent to
		\begin{equation*}
			\frac{\partial g}{\partial Z} (0,\ldots,0,z_0) \neq 0
		\end{equation*}
		for all $ z_0 $ satisfying $ g(0,\ldots,0,z_0) = 0 $.
		Thus we can apply the Implicit Function Theorem \ref{p5-thm:implicitfunc} which gives a power series $ f(X_1,\ldots,X_r) $ with algebraic coefficients such that for $ n $ large enough we have
		\begin{equation*}
			z_n = f(\gamma_1^n,\ldots,\gamma_r^n).
		\end{equation*}
		
		Now we aim to apply Theorem \ref{p5-thm:infinmanycosets}.
		Therefore let us check the conditions of this theorem.
		We have a power series $ f(X_1,\ldots,X_r) $ with algebraic coefficients converging in a neighborhood of the origin with respect to the standard absolute value $ \abs{\cdot}_\nu = \abs{\cdot} $ in $ \CC $.
		The set $ S $ of finitely many absolute values including all archimedean ones is already given by the theorem we are going to prove.
		Define the vector $ \xbf_n $ to be $ (\gamma_1^{w_n},\ldots,\gamma_r^{w_n}) $ where $ w_n $ is the $ n $-th index (in increasing order) within the set $ W $.
		Then $ \xbf_n $ tends to zero, and $ f(\xbf_n) = z_{w_n} $ is defined and belongs to $ K $.
		
		It remains to check the four conditions involving heights. Since $ x_{ni} = \gamma_i^{w_n} $ is an $ S $-unit, we have $ h_S(x_{ni}) + h_S(x_{ni}^{-1}) = 0 $. Thus the first condition is satisfied.
		Consider now
		\begin{equation*}
			\widehat{h}((\gamma_1^{w_n},\ldots,\gamma_r^{w_n})) = \sum_{i=1}^{r} h(\gamma_i^{w_n}) = w_n \cdot \sum_{i=1}^{r} h(\gamma_i) = w_n \cdot C_3
		\end{equation*}
		as well as
		\begin{equation*}
			-\log \left( \max_{i=1,\ldots,r} \abs{\gamma_i^{w_n}} \right) = w_n \cdot \left( -\log \left( \max_{i=1,\ldots,r} \abs{\gamma_i} \right) \right) = w_n \cdot C_4
		\end{equation*}
		which yields $ \widehat{h}(\xbf_n) = O(-\log (\max_i \abs{x_{ni}}_{\nu})) $.
		For the third condition remember that $ z_{w_n} $ is an $ S $-integer (note that this still holds in the other case where we consider $ \widetilde{z}_{w_n} $).
		So for all valuations $ \mu \notin S $ we get $ \abs{z_{w_n}}_{\mu} \leq 1 $ and in the next step $ h_S(z_{w_n}) = 0 $. This implies $ h_S(f(\xbf_n)) = o(h(\xbf_n)) $.
		For the last condition recall that by inequality \eqref{p5-eq:boundhzn} we have
		\begin{equation*}
			h(f(\xbf_n)) \leq C_1 + C_2 \cdot h(1:\gamma_1^{w_n}:\ldots:\gamma_r^{w_n})
		\end{equation*}
		and thus the required bound holds.
		
		Therefore we can apply Theorem \ref{p5-thm:infinmanycosets} which gives us finitely many cosets $ \ubf_1H_1, \ldots,$ $ \ubf_tH_t \subseteq \GG_{\m}^r $ such that $ \set{(\gamma_1^{w_n},\ldots,\gamma_r^{w_n})}_{n \in \NN} $ $ \subseteq \bigcup_{i=1}^{t} \ubf_iH_i $ and such that, for $ i=1,\ldots,t $, the restriction of $ f $ to $ \ubf_iH_i $ coincides with a polynomial $ P_i $ in $ K[X_1,\ldots,X_r] $.
		Hence for all $ n \in W $ there exists an index $ i $ such that $ (\gamma_1^n,\ldots,\gamma_r^n) \in \ubf_iH_i $ and $ z_n = P_i(\gamma_1^n,\ldots,\gamma_r^n) $.
		
		As already noted at the beginning of this proof, the situation of the second case where $ a_0(0,\ldots,0) = 0 $ and $ \widetilde{g}(0,\ldots,0,\widetilde{Z}) $ has only simple zeros is handled in the same way.
		We start again with an infinite sequence $ ((n,z_n))_{n \in W} $ of solutions of the equation
		\begin{equation*}
			g(\gamma_1^n,\ldots,\gamma_r^n,z) = 0
		\end{equation*}
		in $ (n,z) \in \NN \times \Oc_S $ with $ z \neq 0 $, where $ W $ is an infinite subset of $ \NN $.
		The sequence now transforms under the transformation $ \widetilde{z} = a_0(\gamma_1^n,\ldots,\gamma_r^n) z $ to a sequence $ ((n,\widetilde{z}_n))_{n \in W} $ of solutions of
		\begin{equation*}
			\widetilde{g}(\gamma_1^n,\ldots,\gamma_r^n,\widetilde{z}) = 0.
		\end{equation*}
		From here on the steps are the same as in the previous case written down in detail.
	\end{proof}
	
	In the proof of Theorem \ref{p5-thm:solofpolyexp} we first handle some special cases which are not covered by Theorem \ref{p5-thm:mainresult}. After that Theorem \ref{p5-thm:mainresult} can be applied and it remains to classify the output in the intended framework. This is done by distinguishing some cases where in one of them the Hadamard Quotient Theorem is used.
	
	\begin{proof}[Proof of Theorem \ref{p5-thm:solofpolyexp}]
		As above we will denote by $ a_i(n) = a_i(\gamma_1^n,\ldots,\gamma_r^n) $ the coefficients of the polynomial $ g $.
		
		First note that for a fixed value of $ n $ the considered equation
		\begin{equation}
			\label{p5-eq:eqinquestion}
			g(\gamma_1^n,\ldots,\gamma_r^n,z) = 0
		\end{equation}
		has either only finitely many solutions $ z $ if not all $ a_i(n) $ are zero, or holds for all values of $ z $ if all the $ a_i(n) $ are zero.
		Thus for finitely many values of $ n $ the solutions $ (n,z) \in \NN \times \Oc_S $ of \eqref{p5-eq:eqinquestion} having the first component among that finite set can be classified according to the required pattern.
		Therefore we can always assume that $ n $ is large enough.
		
		For $ z = 0 $ equation \eqref{p5-eq:eqinquestion} reduces to $ a_d(n) = 0 $ which has only finitely many solutions in $ n $ since this linear recurrence is non-degenerate by assumption.
		So we can assume in the sequel that $ z \neq 0 $.
		
		It remains to classify the solutions of \eqref{p5-eq:eqinquestion} of the form $ (n,z) \in \NN \times \Oc_S $ with $ z \neq 0 $ and $ n $ large.
		Here we are able to apply Theorem \ref{p5-thm:mainresult} and get finitely many cosets $ \ubf_1H_1, \ldots,$ $ \ubf_tH_t \subseteq \GG_{\m}^r $ as well as for each coset $ \ubf_iH_i $ a polynomial $ P_i $ such that for all remaining solutions $ (n,z) $ of \eqref{p5-eq:eqinquestion} there is an index $ i \in \set{1,\ldots,t} $ with the property that either
		\begin{equation*}
			z = P_i(\gamma_1^n,\ldots,\gamma_r^n) \qquad \text{or} \qquad z = \frac{P_i(\gamma_1^n,\ldots,\gamma_r^n)}{a_0(\gamma_1^n,\ldots,\gamma_r^n)}.
		\end{equation*}
		
		For each $ i = 1,\ldots,t $ we have now four possible situations.
		First, there could exist only finitely many solutions $ (n,z) $ of \eqref{p5-eq:eqinquestion} satisfying $ z = P_i(\gamma_1^n,\ldots,\gamma_r^n) $. They will be contained in $ M $. This is also the case if there exist only finitely many solutions $ (n,z) $ of \eqref{p5-eq:eqinquestion} satisfying $ z = P_i(\gamma_1^n,\ldots,\gamma_r^n) / a_0(\gamma_1^n,\ldots,\gamma_r^n) $.
		
		In the third case we have for a fixed index $ i $ infinitely many solutions $ (n,z) $ of \eqref{p5-eq:eqinquestion} fulfilling the equation $ z = P_i(\gamma_1^n,\ldots,\gamma_r^n) $.
		Putting $ z = P_i(\gamma_1^n,\ldots,\gamma_r^n) $ into equation \eqref{p5-eq:eqinquestion} yields a linear recurrence sequence that is forced to be zero.
		By the theorem of Skolem-Mahler-Lech the set of zeros of this linear recurrence is a finite set together with finitely many arithmetic progressions.
		The finite set contributes to $ M $, whereas the arithmetic progressions $ \Pc_{i,j} $ represent the solutions $ \set{(n,P_i(\gamma_1^n,\ldots,\gamma_r^n)) : n \in \Pc_{i,j}, P_i(\gamma_1^n,\ldots,\gamma_r^n) \in \Oc_S} $. Note that $ P_i(\gamma_1^n,\ldots,\gamma_r^n) $ is a linear recurrence sequence.
		
		For the last possible case assume that for a fixed index $ i $ infinitely many solutions $ (n,z) $ of \eqref{p5-eq:eqinquestion} satisfy the equation $ z = P_i(\gamma_1^n,\ldots,\gamma_r^n) / a_0(\gamma_1^n,\ldots,\gamma_r^n) $.
		Now we plan to apply the above mentioned variant of the Hadamard Quotient Theorem \ref{p5-thm:hadamardquotient}. Let us therefore check the conditions for applying it.
		The recurrences occurring in numerator and denominator are both simple. By partitioning $ \NN $ into finitely many arithmetic progressions and reparametrizing we can assume that the characteristic roots generate a torsion-free multiplicative group.
		It is well known that the ring of $ S $-integers $ \Oc_S $ is finitely generated (consider for instance the decomposition $ \Oc_S = \Oc_K \cdot \Oc_S^* $ and recall that the sets of algebraic integers $ \Oc_K $ and of $ S $-units $ \Oc_S^* $ are both finitely generated) and we have $ z \in \Oc_S $.
		Thus the Hadamard Quotient Theorem can be applied and yields the existence of a linear recurrence $ \ell(n) $ such that $ z = \ell(n) $ for our infinitely many $ n $ (finitely many sequences if partitioning was necessary).
		From here on the procedure is the same as in case three.
	\end{proof}
	
	The proof of our third theorem uses a similar strategy as the proof of Theorem \ref{p5-thm:mainresult} above.
	
	\begin{proof}[Proof of Theorem \ref{p5-thm:red}]
		Let $ K, g, \gamma_1, \ldots, \gamma_r $ be as in the theorem. We write
		\begin{align*}
			g(\gamma_1^n,\ldots,\gamma_r^n,Z) &= Z^d + a_1(\gamma_1^n,\ldots,\gamma_r^n) Z^{d-1} + \cdots + a_d(\gamma_1^n,\ldots,\gamma_r^n) \\
			&= Z^d + a_1(n) Z^{d-1} + \cdots + a_d(n).
		\end{align*}
		We have fixed a finite set $ S $ of places of $ K $, containing all archimedean ones, and such that $ \gamma_1,\ldots,\gamma_r $ and all non-zero coefficients of $ a_i(X_1,\ldots,X_r) $ for $ i=1,\ldots,d $ are $ S $-units.
		Then all coefficients of $ g(\gamma_1^n,\ldots,\gamma_r^n,Z) $ as a polynomial in $ Z $ are $ S $-integers, i.e. have $ \nu $-valuation at most $ 1 $ for all $ \nu \notin S $.
		
		Now we assume that $ g(\gamma_1^n,\ldots,\gamma_r^n,Z) $ is reducible in $ K[Z] $ for infinitely many $ n \in \NN $.
		Denote an infinite set of such $ n $ by $ W_1 $.
		We have to prove that $ g(\gamma_1^n,\ldots,\gamma_r^n,Z) $ has a generic factorization along an arithmetic progression as stated in the theorem.
		Clearly, the other direction of Theorem \ref{p5-thm:red} is trivial.
		
		All solutions $ (n,z) \in \NN \times \overline{\QQ} $ of $ g(\gamma_1^n,\ldots,\gamma_r^n,z) = 0 $ are contained in $ d $ sequences of the shape $ ((n,z_n))_{n \in \NN} $.
		
		As in the proof of Theorem \ref{p5-thm:mainresult} it is shown that the inequality
		\begin{equation*}
			\abs{z_n} \leq \max \setb{1, \abs{a_1(n)} + \cdots + \abs{a_d(n)}}
		\end{equation*}
		holds which implies the boundedness of all $ d $ sequences $ (z_n) $ since $ \abs{\gamma_i} < 1 $.
		Again we can deduce from this boundedness that the $ z_n $ lie in the union of arbitrary small neighborhoods of the solutions of $ g(0,\ldots,0,z) = 0 $ for $ n $ large enough.
		Thus we can split the sequences into finitely many subsequences and consider in what follows only infinite sequences $ (n,z_n) $ which converge to a solution $ z_* $ of $ g(0,\ldots,0,z) = 0 $.
		
		Furthermore, note that for any $ z_n $ we can find a finite extension of $ K $ which contains this element $ z_n $.
		Hence we get a bound
		\begin{equation}
			\label{p5-eq:boundheight}
			h(z_n) \leq C_1 + C_2 \cdot h(1:\gamma_1^n:\ldots:\gamma_r^n).
		\end{equation}
		with absolute constants $ C_1, C_2 $ in the same way as in the proof of Theorem \ref{p5-thm:mainresult}.
		
		Remember that the condition that $ g(0,\ldots,0,Z) $ has no multiple zero as a polynomial in $ Z $ is equivalent to
		\begin{equation*}
			\frac{\partial g}{\partial Z} (0,\ldots,0,z_0) \neq 0
		\end{equation*}
		for all $ z_0 $ satisfying $ g(0,\ldots,0,z_0) = 0 $.
		Thus we can apply the Implicit Function Theorem \ref{p5-thm:implicitfunc} which gives a power series $ f(X_1,\ldots,X_r) $ with algebraic coefficients such that for $ n $ large enough we have
		\begin{equation*}
			z_n = f(\gamma_1^n,\ldots,\gamma_r^n).
		\end{equation*}
		
		Putting together the things we have so far yields that all zeros $ z $ of $ g(\gamma_1^n,\ldots,\gamma_r^n,z) $ can be described by finitely many power series $ f(\gamma_1^n,\ldots,\gamma_r^n) $.
		Since there are only finitely many possible combinations of the finitely many power series, for an infinite subset $ W_2 \subseteq W_1 $ we have the same fixed combination.
		Thus we have
		\begin{equation*}
			g(\gamma_1^n,\ldots,\gamma_r^n,Z) = (Z - f_1(\gamma_1^n,\ldots,\gamma_r^n)) \cdots (Z - f_d(\gamma_1^n,\ldots,\gamma_r^n))
		\end{equation*}
		for all $ n \in W_2 $.
		By our assumption, for all $ n \in W_2 $ we can group together the factors on the right hand side of the last equation such that we obtain two polynomials with coefficients in $ K $.
		For an infinite subset  $ W_3 \subseteq W_2 $ the two polynomials are built in the same way because there are only finitely many possibilities.
		Therefore for all $ n \in W_3 $ we have
		\begin{equation*}
			g(\gamma_1^n,\ldots,\gamma_r^n,Z) = h_1(n,Z) h_2(n,Z)
		\end{equation*}
		with fixed monic polynomials $ h_1(n,Z), h_2(n,Z) $ in $ Z $ having power series of the form $ f(\gamma_1^n,\ldots,\gamma_r^n) $ as coefficients.
		
		Now we take a closer look at the power series $ f(\gamma_1^n,\ldots,\gamma_r^n) $ occurring as coefficients of $ h_1(n,Z), h_2(n,Z) $.
		For any $ n \in W_3 $ the value of $ f(\gamma_1^n,\ldots,\gamma_r^n) $ lies in $ K $.
		Define the vector $ \xbf_n $ to be $ (\gamma_1^{w_n},\ldots,\gamma_r^{w_n}) $ where $ w_n $ is the $ n $-th index (in increasing order) within the set $ W_3 $.
		Then $ \xbf_n $ tends to zero, and $ f(\xbf_n) $ is defined and belongs to $ K $.
		We are going to apply Theorem \ref{p5-thm:infinmanycosets}. So we have to check that the four conditions are satisfied.
		Conditions one and two are checked in the same way as in the proof of Theorem \ref{p5-thm:mainresult}.
		
		The third condition is again satisfied if we can show that $ f(\gamma_1^n,\ldots,\gamma_r^n) $ is an $ S $-integer.
		For a valuation $ \nu $ on $ K $ we define the \emph{Gauss norm} $ \abs{h}_{\nu} $ of a polynomial $ h $ as the maximal $ \nu $-norm of its coefficients.
		By Lemma 1.6.3 in \cite{bombieri-gubler-2006} for two polynomials $ h_1, h_2 $ the equality $ \abs{h_1h_2}_{\nu} = \abs{h_1}_{\nu} \abs{h_2}_{\nu} $ holds for any non-archimedean valuation $ \nu $.
		Since $ g(\gamma_1^n,\ldots,\gamma_r^n,Z) $ is monic as a polynomial in $ Z $ and all its coefficients are $ S $-integers, we have $ \abs{g(\gamma_1^n,\ldots,\gamma_r^n,Z)}_{\nu} = 1 $ for all $ \nu \notin S $.
		Since all archimedean valuations are contained in $ S $ we get
		\begin{equation*}
			1 = \abs{g(\gamma_1^n,\ldots,\gamma_r^n,Z)}_{\nu} = \abs{h_1(n,Z)}_{\nu} \abs{h_2(n,Z)}_{\nu}
		\end{equation*}
		for all $ \nu \notin S $.
		Moreover, we have $ \abs{h_1(n,Z)}_{\nu} \geq 1 $ and $ \abs{h_2(n,Z)}_{\nu} \geq 1 $ because both polynomials are monic.
		Hence they must satisfy $ \abs{h_1(n,Z)}_{\nu} = 1 $ and $ \abs{h_2(n,Z)}_{\nu} = 1 $ and thus have $ S $-integer coefficients.
		
		For the last condition we remember that the $ f(\gamma_1^n,\ldots,\gamma_r^n) $ are sums of products of some $ z_n $. Using Lemma \ref{p5-lemma:heightprop} and the bound \eqref{p5-eq:boundheight} we get
		\begin{equation*}
			h(f(\gamma_1^n,\ldots,\gamma_r^n)) \leq C_3 + C_4 \cdot h(1:\gamma_1^n:\ldots:\gamma_r^n)
		\end{equation*}
		and the required bound holds.
		
		Therefore we can apply Theorem \ref{p5-thm:infinmanycosets} to the coefficients of $ h_1(n,Z), h_2(n,Z) $, which states that these coefficients coincide with polynomials of the form $ P(\gamma_1^n,\ldots,\gamma_r^n) $.
		Thus for an infinite subset $ W_4 \subseteq W_3 $ we have for all $ n \in W_4 $ the factorization
		\begin{equation}
			\label{p5-eq:generic}
			g(\gamma_1^n,\ldots,\gamma_r^n,Z) = h_1(n,Z) h_2(n,Z)
		\end{equation}
		with fixed monic polynomials $ h_1(n,Z), h_2(n,Z) $ in $ Z $ having polynomials of the form $ P(\gamma_1^n,\ldots,\gamma_r^n) $ as coefficients;
		hence the coefficients are linear recurrence sequences.
		
		It remains to prove the part about the arithmetic progression.
		For doing so we rewrite equation \eqref{p5-eq:generic} as
		\begin{equation*}
			g(\gamma_1^n,\ldots,\gamma_r^n,Z) - h_1(n,Z) h_2(n,Z) = 0.
		\end{equation*}
		This equation holds for infinitely many values of $ n $.
		Then we expand the polynomial on the left hand side of the last displayed equation. We get a polynomial in $ Z $ with linear recurrences as coefficients.
		Since this polynomial must be the zero polynomial, the coefficients of all monomials $ Z^i $ must vanish for the infinitely many values of $ n $.
		By the theorem of Skolem-Mahler-Lech there is at least one arithmetic progression $ \Pc $ such that all coefficients vanish for any $ n \in \Pc $.
		Thus equation \eqref{p5-eq:generic} holds for all $ n \in \Pc $.
    \end{proof}
	
	\section{Acknowledgement}
	
	The authors are grateful to Umberto Zannier for useful comments during preparation of this paper in particular related to Theorem \ref{p5-thm:red}.


\begin{thebibliography}{99}
        \bibitem{bombieri-gubler-2006}
        	\textsc{E. Bombieri and W. Gubler},
        	Heights in Diophantine geometry,
        	New Mathematical Monographs, 4, Cambridge University Press, Cambridge, 2006.
        \bibitem{corvaja-zannier-1998}
            \textsc{P. Corvaja and U. Zannier},
            Diophantine equations with power sums and universal Hilbert sets,
            \textit{Indag. Math. (N.S.)} \textbf{3} (1998), 317-332.
        \bibitem{corvaja-zannier-2000}
            \textsc{P. Corvaja and U. Zannier},
            On the Diophantine equation $ f(a^m,y)=b^n $,
            \textit{Acta Arith.} \textbf{94} (2000), 25-40.
        \bibitem{corvaja-zannier-2002-1}
            \textsc{P. Corvaja and U. Zannier},
            Some new applications of the subspace theorem,
            \textit{Compositio Math.} \textbf{131} (2002), 319-240.
        \bibitem{corvaja-zannier-2002-2}
			\textsc{P. Corvaja and U. Zannier},
			Finiteness of integral values for the ratio of two linear recurrences,
			\textit{Invent. Math.} \textbf{149} (2002), 431-451.
		\bibitem{corvaja-zannier-2005}
			\textsc{P. Corvaja and U. Zannier},
			$ S $-unit points on analytic hypersurfaces,
			\textit{Ann. Scient. Ec. Norm. Sup. (4)} \textbf{38} (2005), 76-92.
        \bibitem{corvaja-zannier-2018}
            \textsc{P. Corvaja and U. Zannier},
            Applications of Diophantine approximation to integral points and transcendence,
            Cambridge Tracts in Mathematics, 212, Cambridge University Press, Cambridge, 2018.
		\bibitem{ferretti-zannier-2007}
			\textsc{A. Ferretti and U. Zannier},
			Equations in the Hadamard ring of rational functions,
			\textit{Ann. Sc. Norm. Super. Pisa Cl. Sci. (5)} \textbf{6} (2007), 457-475.
		\bibitem{fuchs-mantova-zannier-2018}
			\textsc{C. Fuchs, V. Mantova and U. Zannier},
			On fewnomials, integral points, and a toric version of Bertini's theorem,
			\textit{J. Amer. Math. Soc.} \textbf{31} (2018), 107-134.
		\bibitem{fuchs-scremin-2004}
			\textsc{C. Fuchs and A. Scremin},
			Polynomial-exponential equations involving several linear recurrences,
			\textit{Publ. Math. Debrecen} \textbf{65/1-2} (2004), 149-172.
		\bibitem{krantz-parks-2002-1}
			\textsc{S. G. Krantz and H. R. Parks},
			A Primer of Real Analytic Functions,
			Second Edition, Birkhäuser, Boston, 2002.
		\bibitem{krantz-parks-2002-2}
			\textsc{S. G. Krantz and H. R. Parks},
			The Implicit Function Theorem: History, Theory, and Applications,
			Birkhäuser, Boston, 2002.
        \bibitem{lang-1983}
            \textsc{S. Lang},
            Fundamentals of Diophantine Geometry,
            Springer, New York, 1983.
		\bibitem{schmidt-1999}
			\textsc{W. M. Schmidt},
			The zero multiplicity of linear recurrence sequences,
			\textit{Acta Math.} \textbf{182} (1999), 243-282.
		\bibitem{schmidt-2003}
			\textsc{W. M. Schmidt},
			Linear recurrence sequences,
			\textit{Diophantine approximation} (Cetraro, 2000), 171-247,
			Lecture Notes in Math. \textbf{1819}, Springer, Berlin, 2003.
		\bibitem{zannier-2009}
			\textsc{U. Zannier},
			Lecture Notes on Diophantine Analysis (with an appendix by Francesco Amoroso),
			Edizioni Della Normale, SNS Pisa, 2009.
        \bibitem{zannier-2010}
            \textsc{U. Zannier},
            Hilbert irreducibility above algebraic groups,
            \textit{Duke Math. J.} \textbf{153} (2010), 397-425.
	\end{thebibliography}
\end{document}